\documentclass[10pt, a4paper]{amsart}

\usepackage[all]{xypic,xcolor}

\usepackage{mathrsfs}
\usepackage{rotating}
\usepackage{url}

\newcommand{\Hom}[2]{\mathrm{Hom}\left({#1},{#2}\right)}
\newcommand{\Z}{\mathbb{Z}}
\newcommand{\R}{\mathbb{R}}
\newcommand{\C}{\mathbb{C}}
\newcommand{\Obj}[1]{\text{Obj}\left({#1}\right)}
\newcommand{\Mor}[1]{\text{Mor}\left({#1}\right)}
\newcommand{\pt}{\{*\}}

\newcommand{\ud}{\mathrm{d}}

\newcommand{\Lattice}{\Lambda}
\newcommand{\dual}[1]{\hat{#1}}
\newcommand{\Hcheck}{\mathcal{H}}

\newcommand{\diffEl}[1]{\mathcal{#1}}
\newcommand{\Forms}{\Omega}
\newcommand{\CS}[2]{\mathrm{CS}\left({#1},{#2}\right)}
\DeclareMathOperator{\Curv}{Curv}
\newcommand{\twist}{\tau}
\newcommand{\poincare}{\mathscr{P}}

\DeclareMathOperator{\im}{im}
\DeclareMathOperator{\ch}{Ch}
\DeclareMathOperator{\vol}{Vol}
\newcommand{\kmap}{{\delta}}
\newcommand{\formsmap}{{i}}
\newcommand{\diffGroup}[1]{\check{#1}}
\newcommand{\diffTop}[1]{\mathbf{\diffGroup{#1}}}
\newcommand{\gTwist}{\mathfrak{Twist}_{\diffGroup{K}}}
\newcommand{\Twist}{\mathfrak{Twist}_K}

\newtheorem{theorem}{Theorem}[section]

\newtheorem{lemma}[theorem]{Lemma}
\newtheorem{proposition}[theorem]{Proposition}
\theoremstyle{definition}
\newtheorem{definition}[theorem]{Definition}
\theoremstyle{remark}
\newtheorem{remark}[theorem]{Remark}
\numberwithin{equation}{section}

\begin{document}

\title{$T$-duality and differential $K$-theory}
\author{Alexander Kahle}
\address{Mathematisches Institut\\Georg-August Universit\"at G\"ottingen}
\email{kahle@uni-math.gwdg.de}
\thanks{The first author was supported by a grant  from the German Research Foundation (Deutsche Forschungsgemeinschaft (DFG)). The second author was supported by the German Research Foundation (Deutsche Forschungsgemeinschaft (DFG)) through the 
Institutional Strategy of the University of G\"ottingen. }
\author{Alessandro Valentino}
\address{Courant Research Centre ``Higher Order Structures''\\Mathematisches Institut\\Georg-August Universit\"at G\"ottingen}
\email{sandro@uni-math.gwdg.de}
\begin{abstract}
We give a precise formulation of $T$-duality for Ramond-Ramond fields. This gives a \emph{canonical} isomorphism between the ``geometrically invariant'' subgroups of the twisted \emph{differential} $K$-theory groups associated to certain principal torus bundles. Our result combines topological $T$-duality with the Buscher rules found in physics.
\end{abstract}
\maketitle
\section{Introduction}\label{sec:introduction}
\subsection{$T$-duality in physics}\label{subsec:tdualphysics}
Much interesting mathematics has come from the interaction between geometry, topology and superstring theory. In this paper we focus our attention on a tool which is believed to be of fundamental importance in superstring theory: \emph{$T$-duality}.\footnote{For a general overview of $T$-duality in superstring theory, see \cite{Polchinski} and \cite{clifford} and references therein.}

$T$-duality is a physical statement asserting that the type IIA-  and type  IIB-superstring theories constructed on pairs of spacetime manifolds possessing isometric torus actions are in a certain sense equivalent. It first appeared as a low energy relation  between type IIA-superstring theory on $\mathbb{R}^{d}\times{S_{r}^{1}}$ and type IIB on $\mathbb{R}^{d}\times{S_{1/r}^{1}}$, where $r$ denotes the radius of the circle. In this simple, or ``untwisted'' case, $T$-duality corresponds to sending $r$ to $1/r$ in the string theory $\sigma$-model.\footnote{In this paper we set the string coupling constant $\alpha'$ to 1.} 

Introducing \emph{$B$-fields} allows one to consider topologically non-trivial situations. In this generality, the spacetime manifolds are total spaces of (usually topologically inequivalent) principal torus bundles over a common base. $T$-duality then asserts that there is a canonical ``equivalence'' between the type IIA theory in the presence of the $B$-field on the first spacetime and the type IIB theory on the second spacetime in the presence of its $B$-field, as long as a particular condition on the geometry of the bundles and their $B$-fields is satisfied. Such pairs of space-times and their $B$-fields are said to be \emph{$T$-dual}. Even more general settings have been considered, where elements of the pair may be bundles of ``non-commutative'' tori \cite{MaRo}. The local rules implementing the equivalence for the fields arising from the low energy limit of superstring theory are referred to in the physics literature as the \emph{Buscher rules} (see \cite{Buscher, Bergshoeff:1995as, Meessen:1998qm, becker-bergman} and references therein).
\subsection{Topological $T$-duality}\label{subsec:tdualtop}
Witten  realised that $D$-brane charges are classified by \emph{$K$-theory} in the low energy limit of type IIA/B superstring theory. In the presence of a $B$-field, the $K$-theory group classifying the $D$-brane charges becomes twisted\footnote{See \cite{atiyah-segal:twisted1,atiyah-segal:twisted2,FHT1} for details on twisted $K$-theory.} \cite{Witten1998,BM}. $T$-duality is expected to be an equivalence of the low energy theory on $T$-dual pairs and thus, restricting attention to the $D$-brane charges, predicts a canonical isomorphism between the appropriate twisted $K$-theory groups on principal $T$-bundles. This prediction can be made mathematically rigorous, and is known as \emph{topological} $T$-duality. It has been widely investigated and shown to hold \cite{BoEvMat,BoHaMat,BuSchi,BuRuSchi}.

In order to give a feeling for the ingredients involved, we will briefly describe the usual mathematical framework of topological $T$-duality in the simplest situation, namely when the torus bundles are circle bundles.

Let $P\xrightarrow{\pi}X$, $\dual{P}\xrightarrow{\dual{\pi}}X$ be principal $U(1)$-bundles, and and let $\twist$, $\dual{\twist}$ be twists of $K(P)$ and $K(\dual{P})$ respectively. One says $(P,\twist)$ is $T$-dual to $(\dual{P},\dual{\twist})$, when the condition
\begin{equation}\label{eq:introphcriterion}
\pi_{*}[\twist]=c_{1}(\hat{P}),\quad\dual{\pi}_{*}[\dual{\twist}]=c_{1}(P),
\end{equation}  
is satisfied, where $[\twist]$ and $[\dual{\twist}]$ denote the cohomology classes in $H^{3}(-;\Z)$ classifying the twists. A consequence of the condition Eq.~\ref{eq:introphcriterion} is that 
\[
[c_{1}(P)]\smile[c_1(\dual{P})]=0.
\]
Suppose now that the pairs are $T$-dual. Consider the correspondence diagram
\begin{equation}\label{correspondence}
\xymatrix{&\ar[ld]_{\dual{\pi}}P\times_{X}\hat{P}\ar[rd]^{\pi}&\\
P& &\hat{P}}
\end{equation}
and define the homomorphism
\begin{displaymath}
T_K:K^{\bullet+\twist}(P)\to{K^{\bullet+\dual{\twist}-1}(\hat{P})}
\end{displaymath}
by
\begin{equation}\label{intro:tdualeq}
T_K:={\pi}_{*}\circ\Theta_\poincare\circ{\dual{\pi}}^{*},
\end{equation}
where $\Theta_\poincare:K^{\bullet+\dual{\pi}^*\twist}(P\times_{X}\hat{P})\to{K^{\bullet+\pi^*\dual{\twist}}(P\times_{X}\hat{P})}$ is a distinguished isomorphism between the twisted $K$-theory groups that may intuitively be thought of as being induced by multiplying by the ``Poincar\'e line bundle'' (in the sense that on restricting attention to the pre-image of any point in $P$, $\Theta_\poincare$ becomes an automorphism of $K$-theory induced by tensoring by the Poincar\'e line  bundle). The  statement from $T$-duality in this situation is that the map $T_K$ is a \emph{natural isomorphism}. Details may be found in \cite{BoEvMat,BuSchi}.

The form of the isomorphism in \ref{intro:tdualeq} allows one to think of topological $T$-duality being a topological version of the Fourier-Mukai transform. 
\subsection{Differential $K$-theory and $T$-duality}\label{subsec:tdualdiff}
$T$-duality, as understood by physicists, encompasses far more than the equivalence of ``charges'' that underlies topological $T$-duality. Our goal in this paper is to recover more of the physical picture, and in particular, examine how $T$-duality acts on the space of \emph{fields}. 

Just as charges in string theory turn out to be subtle, so too do fields. Physicists have been accustomed to see fields as represented by some local object: a function, say, or differential form; but in the late 1990's it was realised that certain fields in the low energy limits of type IIA/B string-theory, the  Ramond-Ramond fields, have both local and global (i.e. topological) aspects \cite{Fr:Dirac, Moore2000, FrHo}, and a new language was needed to describe them. Fortunately, such a framework existed in mathematics: the language of differential cohomology theories.

Differential cohomology theories are geometric refinements of cohomology theories. For example, one is used to thinking of class in $H^2(X;\Z)$ as representing an isomorphism class of a line bundle on $X$. By contrast, a class in $\diffGroup{H}^2(X;\Z)$, the second \emph{differential} cohomology group of $X$, represents an isomorphism class of a line bundle \emph{with connection}. 

Formally, differential cohomology theories may be thought of as completing the pullback square 
\begin{equation}\label{eq:pbsquare}
\xymatrix{
?\ar[r]\ar[d]&\ar[d]\Forms^\bullet(X;V)\\
E^\bullet(X)\ar[r]^<<<<<c&H^\bullet(X;V)
}
\end{equation}
as \emph{cohomology theories}, where $E^\bullet$ is a generalised cohomology theory, $V=E^\bullet(\{*\})\otimes_\Z\R$, and $c:E^\bullet(X)\to H^\bullet(X;V)$ is a given map (for $K$-theory this is the Chern character). For ordinary cohomology, such theories were first constructed by Cheeger and Simons \cite{CheSim}, and Deligne \cite{De}. The realisation of the relevance of such theories in physics lead to their being constructed for generalised cohomology theories \cite{Fr:Dirac}. The definitive general construction is by Hopkins and Singer \cite{HoSi}. 

The theory of interest for us will be the differential $K$-theory of a space, $\diffGroup{K}^\bullet(X)$, and twists thereof. While Hopkins and Singer provide a general construction for \emph{any} generalised cohomology theory that fits into the square (\ref{eq:pbsquare}), for differential $K$-theory there are much more geometric constructions \cite{bunke-2007, SiSu, Klonoff, CaMiWa}, and one may intuitively regard elements of $\diffGroup{K}^0(X)$ as being a formal differences of vector bundles on $X$ \emph{with connection}.

As hinted at earlier, it was realised that Ramond-Ramond fields are represented by classes\footnote{More precisely: equivalence classes are represented by elements in differential $K$-theory, but the actual fields are represented by co-cycles.} of the differential $K$-theory of a space, and in the presence of $B$-fields, classes of \emph{twisted} differential $K$-theory \cite{Fr:Dirac,FrHo}. $T$-duality would lead one to expect an equivalence between the spaces of fields on dual torus bundles with the appropriate $B$-fields, and this is precisely what we investigate. Understanding $T$-duality at the level of Ramond-Ramond fields has also recently attracted the attention of physicists, for example in the  work of Becker and Bergman \cite{becker-bergman}. 

The main contribution of our paper is a precise mathematical formulation of $T$-duality for Ramond-Ramond fields (which extends a formulation of topological $T$-duality due to Dan Freed\footnote{\emph{Private communication}.}) and proving the existence of a \emph{canonical} $T$-duality isomorphism for these. We characterise what a $T$-duality pair should be in the context of differential $K$-theory, and from such a pair, construct canonical twists of the differential $K$-theory of the torus bundles. Having constructed these, we show that the twisted differential $K$-theory groups on the bundles are isomorphic, with the isomorphism furnished by a differential refinement of the push-pull of topological $T$-duality (Eq.~\ref{intro:tdualeq}). In this sense, one may see our work as a geometric generalisation of topological $T$-duality, leading to a ``differential-geometric'' Fourier-Mukai transform.

From another point of view, our work combines the global aspects of topological $T$-duality with the ``Buscher rules'' which are known to hold for \emph{topologically trivial} Ramond-Ramond fields. This point of view brings an essential new ingredient of our work into focus. Topologically trivial Ramond-Ramond fields can be described by globally defined differential forms. In the simple case described in the previous section, a type IIA (topologically trivial) Ramond-Ramond field has a field-strength given by an element of $G\in\Forms^{\mathrm{ev}}(X\times S^1)$, and a type IIB Ramond-Ramond field has a field-strength $\dual{G}\in\Forms^{\mathrm{odd}}(X\times S^1)$. The Buscher rules then state that $T$-duality takes the field-strength of type IIA Ramond-Ramond fields to that of type IIB Ramond-Ramond fields via the transformation
\begin{equation}\label{eq:Buscher}
G\mapsto \int_{S^1}\dual{\pi}^*{G}\wedge e^F,
\end{equation}
where $F$ is the curvature of the Poincar\'e line bundle, which is usually referred to as the Hori formula \cite{Hori}. 

The crucial point is that the Buscher rules do \emph{not} induce an isomorphism for arbitrary (topologically trivial) Ramond-Ramond fields, but only those that are \emph{invariant} of the circle action. If one wants to understand $T$-duality for Ramond-Ramond fields on general principal $T$-bundles, then, one must understand what the right notion of invariance should be for elements of differential $K$-theory. We find that restricting attention to the fixed classes in differential $K$-theory is insufficient. The correct subgroup, which we call the \emph{geometrically invariant} subgroup, is essentially the subgroup of classes with invariant Chern character form.
\subsection{Organisation}\label{subsec:organisation}
This paper is organised as follows. Section \ref{sec:smoothtdual} contains our main result (Th.~\ref{theorem:tdualdiff}). We give a precise formulation of $T$-duality in \emph{differential} $K$-theory, and prove the $T$-duality isomorphism in this context. A new feature is the necessity of restricting to the right notion of ``invariant'' subgroup. Understanding the action of the torus on the various cohomology groups and twists in $T$-duality is subtle, and we discuss topics relating to this in Appendix \ref{appendix:groupactions}. In Sec.~\ref{subsec:relationtoother}, we relate our approach to $T$-duality to the various approaches to \emph{topological} $T$-duality found in the literature, and to the the physical conception of $T$-duality captured in the Buscher rules. The remaining appendix, Appendix \ref{appendix:differential}, reviews various topics in differential cohomology theories. In particular, we discuss the category structure induced on differential cohomology by ``geometric trivialisation'', which is central to our formulation of $T$-duality in the differential setting. We also discuss the properties of twisted differential $K$-theory that we require, and twists of differential $K$-theory. 
\subsection*{Acknowledgements} The authors are very grateful to Dan Freed for explaining his point of view of $T$-duality to them, and for many fruitful discussions. They would also like to thank Thomas Schick and Ulrich Bunke for useful conversation, and also Aaron Bergman and Ansgar Schneider. The second author would like to thank Richard Szabo, Sara Azzali, and Alessandro Fermi for their interest in the project. 
\section{$T$-duality in differential $K$-theory}\label{sec:smoothtdual}
We give a precise formulation of $T$-duality in differential $K$-theory. Our point of view is inspired by an approach to topological $T$-duality explained to us by Dan Freed. For us, the fundamental datum is a \emph{smooth $T$-duality pair} (Def.~\ref{def:smoothtdualpair}). From this we construct canonical twists (Sec.~\ref{sec:difftwist}) of the differential $K$-theory of the torus bundles, and canonical homomorphism, which we henceforth refer to as the $T$-map, between those twisted differential $K$-theory groups (Eq.~\ref{eq:tmap}). The  $T$-map is not an isomorphism on the entire twisted differential $K$-group, but rather on a subgroup: the \emph{geometrically invariant} subgroup (Def.~\ref{def:geominv}). The main theorem, and in particular the statement of $T$-duality in this setting, is Th.~\ref{theorem:tdualdiff}. 
\subsection{Preliminaries}
Let $V$ be a vector space, with metric. Let $\Lattice\subset V$ be a full lattice. Let $\dual{V}=\Hom{V}{\R}$ and $\dual{\Lattice}=\Hom{\Lattice}{\Z}\subset\dual{V}$ be the duals of $V$ and $\Lattice$ respectively. Form the tori $T=V/\Lattice$ and $\dual{T}=\dual{V}/\dual{\Lattice}$. $T$ and $\dual{T}$ have geometry: they inherit invariant metrics from the inner products on $V$ and $\dual{V}$, and flat connections.

Let $X$ be a smooth manifold, and $\pi:(P,\nabla)\to X$, $\dual{\pi}:(\dual{P},\dual{\nabla})\to X$ be, respectively, principal $T$- and $\dual{T}$-bundles with connection. Let $f:X\to BT$, $\bar{f}:P\to ET$ be classifying maps adapted to the connection such that the following diagram commutes:
\[
\xymatrix{
P\ar[r]^{\bar{f}}\ar[d]^\pi&ET\ar[d]^{\pi^u}\\
X\ar[r]^f&BT
}
\]
and $\dual{f}:X\to B\dual{T}$, $\bar{\dual{f}}:\dual{P}\to E\dual{T}$ classifying maps adapted to $(\dual{P},\dual{\nabla})\to X$. We will consider $P$ and $\dual{P}$ with metrics compatible with the connections as, for instance, in Sec.~5.3 of \cite{MaWu}.  From now on, whenever we refer to $P$ and its dual, we implicitly assume that they carry the classifying maps with them.
 Principal $T$-bundles with connection determine \emph{objects} in $\diffEl{H}^{2}(X;\Lattice)$, the second differential cohomology groupoid of $X$ with values in $\Lattice$.\footnote{The groupoid structure we use throughout is  the \emph{geometric} structure described in Appendix \ref{appendix:sec:geometric}. We briefly describe differential cohomology with values in lattices in Appendix \ref{app:lattices}.}  We will write $\diffEl{P}\in\diffEl{H}^{2}(X;\Lattice)$ and $\diffEl{\dual{P}}\in\diffEl{H}^{2}(X;\dual{\Lattice})$ to the objects determined by $(P,\nabla)$ and $(\dual{P},\dual{\nabla})$.

There is a pairing $\cdot:\Hcheck^k(X;\Lattice)\otimes\Hcheck^l(X;\dual{\Lattice})$ given by the sequence
\begin{equation}\label{eq:dotpairing}
\xymatrix{
\Hcheck^k(X;\Lattice)\otimes\Hcheck^l(X;\dual{\Lattice})\ar[r]^>>>>>{\smile}&\Hcheck^{k+l}(X;\Lattice\otimes\dual{\Lattice})\ar[r]&\Hcheck^{k+l}(X).
}
\end{equation}
We demand a (geometric) trivialisation $\diffEl{\sigma}\in\Mor{\Hcheck^4(X)}$ of $\diffEl{P}\cdot\dual{\diffEl{P}}$: 
\[
\diffEl{\sigma}:0\to\diffEl{P}\cdot\diffEl{\dual{P}},
\]
and in particular that $\diffEl{P}\cdot\diffEl{\dual{P}}$ be in the connected component of $0\in\diffEl{H}^4(X)$. This is \emph{not} the same as demanding that the class of $\diffEl{P}\cdot\diffEl{\dual{P}}$ be zero in the group $\diffGroup{H}^4(X)$.

Taking this together, we arrive at the following definition.
\begin{definition}[Differential $T$-duality pair]\label{def:smoothtdualpair} A \emph{differential $T$-duality pair} over a smooth manifold $X$ is given by the triple $\left[(P,\nabla),(\dual{P},\dual{\nabla}),\sigma\right]$, where $(P,\nabla)$, $(\dual{P},\nabla)\to X$ are respectively principal $T$, $\dual{T}$ bundles over $X$ with connection (and classifying maps), and $\diffEl{\sigma}:0\to\diffEl{P}\cdot{\diffEl{\dual{P}}}$ is an element of $\Mor{\Hcheck^4(X)}$. 
\end{definition}
\subsection{Canonical twistings}\label{sec:difftwist}
The pullback objects $\pi^*\diffEl{P}\in\diffEl{H}^2(P;\Lattice)$, $\dual{\pi}^*\diffEl{\dual{P}}\in\diffEl{H}^2(\dual{P};\Lattice)$ have canonical geometric trivialisations $\delta_{\diffEl{P}}:0\to\pi^*\diffEl{P}$, $\dual{\delta}_{\diffEl{\dual{P}}}:0\to\dual{\pi}^*\diffEl{\dual{P}}$ furnished by the ``diagonal sections"\footnote{For a careful discussion of how a section of a principal torus bundle determines a trivialisation in differential cohomology, refer to Appendix \ref{appendix:sec:actionontwist}.}
\begin{align*}
\Delta_P&:p\mapsto(p,p),\\
\Delta_{\dual{P}}&:\dual{p}\mapsto(\dual{p},\dual{p}).
\end{align*}
 We note that $\delta_\diffEl{P}\cdot\pi^*\dual{\diffEl{P}}\in\Mor{\diffEl{H}^4(P;\Z)}$. Explicitly, 
\[
\delta_\diffEl{P}\cdot\pi^*\dual{\diffEl{P}}:0\to\pi^*\diffEl{P}\cdot\pi^*\dual{\diffEl{P}}.
\] 
Similarly $\dual{\pi}^*\diffEl{P}\cdot\dual{\delta}_{\dual{\diffEl{P}}}:0\to\dual{\pi}^*\diffEl{P}\cdot\dual{\pi}^*\dual{\diffEl{P}}$. Comparing these with $\diffEl{\sigma}$ gives automorphisms of $0\in\diffEl{H}^4(P;\Z)$ (resp. $0\in\diffEl{H}^4(\dual{P};\Z)$):
\begin{align*}
\pi^*\diffEl{\sigma}-\diffEl{\delta}_{\diffEl{P}}\cdot\pi^*\diffEl{\dual{P}}:0\to0,\\
\dual{\pi}^*\diffEl{\dual{\sigma}}-\dual{\pi}^*\diffEl{P}\cdot\diffEl{\dual{\delta}}_{\diffEl{\dual{P}}}:0\to0.
\end{align*}
Automorphisms of $0\in\diffEl{H}^4(-;\Z)$ define objects in $\diffEl{H}^3(-;\Z)$, and in this way we obtain objects $\diffEl{\twist}\in\diffEl{H}^3(P;\Z)$ and $\diffEl{\dual{\twist}}\in\diffEl{H}^3(\dual{P};\Z)$. As explained in Appendix \ref{appendix:sec:twist}, objects in $\diffEl{H}^3(-;\Z)$ determine twists of \emph{differential} $K$-theory.
\subsection{A morphism of twists}
Consider the fibre product
\[
\xymatrix{
&P\times_X\dual{P}\ar[dl]_{\dual{\pi}}\ar[dr]^{\pi}&\\
P\ar[dr]_\pi&&\dual{P}\ar[dl]^{\dual{\pi}}\\
&X&
}
\]
\begin{lemma}\label{lem:diffis} The product $\dual{\pi}^*\delta_{\diffEl{P}}\cdot\pi^*\dual{\delta}_{\diffEl{\dual{P}}}$ is an element of $\Mor{\diffEl{H}^3(X;\Z)}$, and
\begin{equation}\label{eq:diffis}
\dual{\pi}^*\diffEl{\delta}_\diffEl{P}\cdot\pi^*\diffEl{\dual{\delta}}_\diffEl{\dual{P}}:\dual{\pi}^*\twist\to\pi^*\dual{\twist}.
\end{equation}
\end{lemma}
\begin{proof}
We will, for clarity of notation, suppress pullbacks. As explained in Appendix \ref{appendix:sec:geometric}, $\delta_{\diffEl{P}}\in \diffEl{C}^1(1)(P;\Lattice)$ such that $\check{d}\delta_{\diffEl{P}}=\diffEl{P}$, with a similar statement for $\dual{\delta}_{\diffEl{\dual{P}}}$. Then $\delta_{\diffEl{P}}\cdot\dual{\delta}_{\dual{\diffEl{P}}}\in \diffEl{C}^2(2)(P\times_X\dual{P};\Z)$, and
\begin{align*}
\check{d}\left(\delta_{\diffEl{P}}\cdot\dual{\delta}_{\dual{\diffEl{P}}}\right)&=\diffEl{P}\cdot\dual{\delta}_{\dual{\diffEl{P}}}-\delta_{\diffEl{P}}\cdot\dual{\diffEl{P}}\\
&=-(\sigma-\diffEl{P}\cdot\dual{\delta}_{\dual{\diffEl{P}}})+(\sigma-\delta_{\diffEl{P}}\cdot\dual{\diffEl{P}})\\
&=\twist-\dual{\twist}.
\end{align*}
This establishes the lemma.
\end{proof} 
\subsection{The $T$-duality isomorphism}\label{subsec:difftdualiso}
Before stating the main theorem of the paper, we need a definition.
\begin{definition}[Geometric invariant subgroup]\label{def:geominv}
Let $G$ be a compact and connected Lie group, and $X$ be a smooth manifold with a smooth $G$-action. Let $h\in\gTwist(X)$ be such that $\Curv h$ is $G$-invariant. The \emph{geometrically invariant} subgroup of $\diffGroup{K}^{h+\bullet}(X)$ is then defined to be the subgroup
\[
\diffGroup{K}^{h+\bullet}(X)^G\subseteq\diffGroup{K}^{h+\bullet}(X)
\]
 of all $x\in\diffGroup{K}^{h+\bullet}(X)$ such that for any $g\in G$
\[
g^*\Curv x -\Curv x=0,
\]
where $\Curv:\diffGroup{K}^{\twist+\bullet}(X)\to\Forms^{\twist+\bullet}(X)$ is the map in Eq. \ref{eq:diffkexactcurv}.
\end{definition}
\begin{theorem}\label{theorem:tdualdiff}
The sequence
\begin{equation}\label{eq:tmap}
T_{\diffGroup{K}}:\xymatrix{
\diffGroup{K}^{\diffEl{\twist}+\bullet}(P)^T\ar[r]^<<<<<{\dual{\pi}^*}&\diffGroup{K}^{\diffEl{\twist}+\bullet}({P\times_X\dual{P}})\ar[r]^{\Theta_\poincare}&\diffGroup{K}^{\diffEl{\dual{\twist}}+\bullet}({P\times_X\dual{P}})\ar[r]^<<<<<{\pi_*}&\diffGroup{K}^{\diffEl{\dual{\twist}}+\bullet-\dim T}(\dual{P})^{\dual{T}}
}
\end{equation}
is an isomorphism. The middle map, which we denote $\Theta_\poincare:\diffGroup{K}^{\twist+\bullet}(P\times_X\dual{P})\to\diffGroup{K}^{\dual{\twist}+\bullet}(P\times_X\dual{P})$, is the isomorphism of twistings defined by the morphism in Lemma \ref{lem:diffis}.
\end{theorem}
\begin{remark} By ``forgetting the geometry'', one recovers the topological $T$-\!\!~homomorphism, and our statement essentially reduces to the well-known statements of topological $T$-duality found in \cite{BuSchi, BuRuSchi, BoHaMat, BoEvMat}. On the other hand, taking curvatures of the twists, $T_{\diffGroup{K}}$ induces a homomorphism between $\tau$-twisted differential forms on $P$, and $\dual{\tau}$-twisted differential forms on $\dual{P}$. One can show that the $T$-map induced on twisted differential forms is only an isomorphism when restricted to the \emph{invariant} differential forms. The $T$-maps commute with $\Curv$ (up to a factor of $\hat{A}$), and thus the most one can hope for is that the $T$-map on differential $K$-theory be an isomorphism between the geometrically invariant subgroups. This shows that in some sense our statement is the most general statement possible for the twisted $K$-theory groups in question. In Appendix \ref{appendix:groupactions} we see that the fixed subgroups of differential $K$-theory do not surject onto $K$-theory, and thus restricting attention to these would lead to a weaker statement than the physics predicts.
\end{remark}
The details of our proof of Th.~\ref{theorem:tdualdiff} are fairly technical, but the basic idea is simple: we wish use the fact that we know that $T$-duality {holds} for $K$-theory (Prop.~\ref{prop:tdual}) and for invariant forms (Prop.~\ref{prop:tdualforms}). Prop.~\ref{prop:sequence} shows precisely how these relate to the geometrically invariant subgroup of differential $K$-theory, allowing us to form the commutative diagram in Fig.~\ref{diagram:five}. Applying the five lemma now proves the theorem. 

Before we proceed in discussing the detailed proof of Th.~\ref{theorem:tdualdiff}, we introduce some useful notation and an easy lemma.
\begin{definition}[Averaging map] Let $G$ be a compact Lie group with an invariant volume form $\mu$,\footnote{Henceforth we will always assume an invariant measure.} and let $X$ be a smooth $G$-space. Let $m:G\times X\to X$ be the action of $G$, and $\pi:G\times X\to X$ the canonical projection. Let $h\in\gTwist(X)$ with $\Curv h$ $G$-invariant. Define for any $x\in\Forms^{h+\bullet}(X)$,  its average as
\[
\bar{x}:=\frac{1}{\mathrm{vol}\;G}\pi_*(m^*(x)\wedge\mu)={\frac{1}{\mathrm{vol}\;G}\int g^{*}x\;\ud\mu_{g}},
\]
where $\ud\mu_g$ is the measure induced by $\mu$.
\end{definition}
\begin{lemma}\label{lem:averaged}
Let $G$ be a compact Lie group, $X$ a smooth $G$-manifold and $h\in\gTwist(X)$ with $\Curv h$ $G$-invariant. Then, for any $\omega\in\Forms^{h+\bullet}(X)$,
\[
\ud_h\bar{\omega}=\overline{\ud_h\omega}.
\]
In particular, if $H^{h+\bullet}(X)$ is fixed by $G$, and $\omega$ is $\ud_h$-closed, then
\[
\omega=\bar{\omega}+\ud_h\alpha
\]
for some $\alpha\in\Forms^{h+\bullet-1}(X)$.
\end{lemma}
\begin{proof}
The first part is direct computation. Assume now $H^{h+\bullet}(X)$ is fixed by $G$, and let $\omega$ be $\ud_h$-closed. Then one may choose a family $\eta_g\in\Forms^{h+\bullet}(X)$ depending smoothly on $g$ such that
\[
g^*\omega=\omega+\ud_h\eta_{g},
\]
as $g^*\omega$ and $\omega$ are, by assumption, $\ud_h$-cohomologous.
 Averaging this equation over $G$ gives the second part of the lemma.
\end{proof}

\begin{lemma}\label{lem:greg}  Let $G$ be a connected compact Lie group, $X$ and a smooth $G$-manifold. Then $G$ acts trivially on $H^{h+\bullet}(X)$ for any $h\in\gTwist(X)$ with $\Curv h$ $G$-invariant.
\end{lemma}
\begin{proof}
Any smooth map  $f:X\to X$ induces a homomorphism $f^*:H^{h+\bullet}(X)\to H^{f^*h+\bullet}(X)$. For two such maps $f,g:X\to Y$, the induced homomorphisms are equal if there exists a homotopy from $f$ to $g$ holding the twist fixed (see Sec.~1.3 of \cite{MaWu}). Now, for any $h\in\gTwist(X)$ with $\Curv h$ $G$-invariant, i.e. $g^*\Curv h=\Curv h$, one has that $g^*:H^{h+\bullet}(X)\to H^{h+\bullet}(X)$ is the identity morphism, as, by the connectedness of $G$, any $g\in G$ may be connected to the identity by a path in $G$. 
\end{proof} 

We begin with a lemma relating the ``geometrically invariant'' differential $K$-groups to $K$-theory.
\begin{lemma}\label{lem:ksurject} Let $G$ be a compact connected Lie group, $X$ a compact smooth $G$-manifold, and $h\in\gTwist(X)$ such that $\Curv h$ is $G$-invariant. Then the map
\[
\kmap:\diffGroup{K}^{h+\bullet}(X)^G\to K^{h+\bullet}(X)
\]
 is a surjection (where $\kmap$ is the restriction of the corresponding map in Eq.~\ref{eq:diffkexact}).
\end{lemma}
\begin{proof}
Let $x\in K^{h+\bullet}(X)$ and choose any $\tilde{x}\in\kmap^{-1}(x)\subset\diffGroup{K}^{h+\bullet}(X)$. Then, by Lemma \ref{lem:averaged},
\[
\Curv\tilde{x}=\overline{\Curv\tilde{x}}+\ud_h\alpha,
\]
for some $\alpha\in\Forms^{h+\bullet-1}(X)$. We then notice that
\[
\kmap(\tilde{x}-\formsmap(\alpha))=\kmap(\tilde{x})-\kmap(\formsmap(\alpha))=x,
\]
and
\[
\Curv(\tilde{x}-\formsmap(\alpha))=\Curv\tilde{x}-\ud_h\alpha=\overline{\Curv\tilde{x}},
\]
where $\formsmap:\Forms^{h+\bullet-1}(X)\to\diffGroup{K}^{h+\bullet}(X)$ is {related to} the map in Eq.~\ref{eq:diffkexact}.

From the equations above, we see that $\tilde{x}-\formsmap(\alpha)\in\diffGroup{K}^{h+\bullet}(X)^G$, and its image under $\kmap$ is $x$, as required.
\end{proof}
We now wish to understand the kernel of the map $\kmap$ in Lemma \ref{lem:ksurject}.
\begin{lemma}\label{lem:kersurj}
Let $G$ be a compact connected Lie group, $X$ a compact smooth $G$-manifold, and $h\in\gTwist(X)$ such that $\Curv h$ is $G$-invariant. Then any $\omega\in \formsmap^{-1}(\diffGroup{K}^{h+\bullet}(X)^{G})$ is of the form
\[
\omega=\omega'+\ud_h\alpha,
\]
where $\omega'$ is invariant.
\end{lemma}
\begin{proof}
We notice that the sequence
\[
\xymatrix{
\Forms^{h+\bullet-1}(X)\ar[r]^\formsmap&\diffGroup{K}^{h+\bullet}(X)\ar[r]^\Curv&\Forms^{h+\bullet}(X)
}
\]
sends
\[
\omega\mapsto\ud_h\omega.
\]
Thus, $\omega\in \formsmap^{-1}(\diffGroup{K}^{h+\bullet}(X)^{G})$ iff for all $g\in G$
\[
g^*\ud_h\omega=\ud_h\omega.
\]
Averaging this equation we see
\[
\bar{\omega}=\omega+\eta
\]
for some $\ud_h$-closed differential form $\eta$. Averaging again shows $\bar{\eta}=0$, so that, by Lemma \ref{lem:averaged}, $\eta$ is exact.
\end{proof}
One more lemma will allow us to understand $i$ completely.
\begin{lemma}\label{lem:keri}
Let $G$ be a compact connected Lie group, $X$ a compact smooth $G$-manifold, and $h\in\gTwist(X)$ such that $\Curv h$ is $G$-invariant. Then
\begin{enumerate}
\item $\formsmap:\Forms^{h+\bullet-1}(X)\to\diffGroup{K}^{h+\bullet}(X)$ descends to a map 
\[
\formsmap:\frac{\Forms^{h+\bullet-1}(X)}{\ud_h\Forms^{h+\bullet}(X)}\to\diffGroup{K}^{h+\bullet}(X).
\]
\item The kernel of
\[
\formsmap:\frac{\Forms^{h+\bullet-1}(X)}{\ud_h\Forms^{h+\bullet}(X)}\to\diffGroup{K}^{h+\bullet}(K),
\]
is a lattice, and invariant under the action of $G$. Explicitly, it is
\[
\frac{\Forms^{h+\bullet-1}_{\im\ch}(X)}{\ud_h\Forms^{h+\bullet}(X)}.
\]
\end{enumerate}
\end{lemma}
\begin{proof}
Point 1 follows by noting that $\ud_h$-exact forms are in the kernel of $i$ by {the properties of twisted differential $K$-theory described in Appendix \ref{appendix:sec:twist}}. Point 2 holds for similar reasons -- any two elements in a component of $\ker \formsmap$ are cohomologous, so differ by a $\ud_h$-exact differential form. By the naturality of $\formsmap$, the $\ker\formsmap$ is mapped to itself under the action of $G$. 
\end{proof}
We may gather all the above lemmas together into one proposition.
\begin{proposition}\label{prop:sequence}  Let $G$ be a compact connected Lie group, $X$ a compact smooth $G$-manifold, and $h\in\gTwist(X)$ such that $\Curv h$ is $G$-invariant.
The sequence 
\begin{equation}\label{eq:exactseq}
0\to\frac{\Forms^{h+\bullet-1}_{\im\ch}(X)}{\ud_h\Forms^{h+\bullet}(X)}
\to\left(\frac{\Forms^{h+\bullet-1}(X)}{\ud_h\Forms^{h+\bullet}(X)}\right)^G
\xrightarrow{\formsmap}
\diffGroup{K}^{h+\bullet}(X)^G
\xrightarrow{\kmap}
K^{h+\bullet}(X)\to0
\end{equation}
is exact. 
\end{proposition}
\begin{remark} It is not hard to see that
\[
\frac{\left(\frac{\Forms^{h+\bullet}(X)}{\ud_h\Forms^{h+\bullet-1}(X)}\right)^G}{\frac{\Forms^{h+\bullet}_{\im\ch}(X)}{\ud_h\Forms^{h+\bullet-1}(X)}}\cong\left(\frac{\Forms^{h+\bullet}(X)}{\Forms^{h+\bullet}_{\im\ch}(X)}\right)^G,
\] 
and that this isomorphism commutes with all the maps of interest, allowing us, for instance, to re-write the sequence above in a more familiar form:
\[
\xymatrix{
0\ar[r]&\left(\frac{\Forms^{h+\bullet-1}(X)}{\Forms^{h+\bullet-1}_{\im\ch}(X)}\right)^G\ar[r]^\formsmap&\diffGroup{K}^{h+\bullet}(X)^G\ar[r]^\kmap&
K^{h+\bullet}(X)\ar[r]&
0
}.
\]
However, it seems the less familiar form is more natural in our proof, so we prefer to write things that way.
\end{remark}

We now need to understand how $T$-duality acts on differential forms, and on $K$-theory. We begin with differential forms.
\begin{definition}
We denote by $T_\Forms:\Forms^{\twist+\bullet}(P)\to\Forms^{\dual{\twist}+\bullet-\dim{T}}(\dual{P})$ the map 
\[
T_\Forms:\omega\mapsto\pi_*\left(\exp\poincare\wedge\dual{\pi}^*\omega\right),
\]
where $\poincare=\Curv\dual{\pi}^*\delta_{\diffEl{P}}\cdot\pi^*\dual{\delta}_{\dual{\diffEl{P}}}$. 
\end{definition}
One may see that 
\[
i\circ T_\Forms=T_{\diffGroup{K}}\circ i,
\]
where $i$ is as in Eq. \ref{eq:diffkexact}.

\begin{lemma}\label{lem:tclosed} The map $T_\Forms:\Forms^{\twist+\bullet}(P)\to\Forms^{\dual{\twist}+\bullet-\dim{T}}(\dual{P})$ sends $\ud_\twist$-exact forms to $\ud_{\dual{\twist}}$-exact forms.
\end{lemma}
\begin{proof}
$\poincare$ has the property 
\[
\ud_{\dual{\twist}}\left(\exp\poincare\wedge\omega\right)=\exp\poincare\wedge\ud_\twist\omega
\]
as a consequence of the fact that $\dual{\pi}^*\delta_{\diffEl{P}}\cdot\pi^*\delta_{\dual{\diffEl{P}}}:\twist\to\dual{\twist}$. This, along with the explicit form for $T_\Omega$, now shows that
\[
T_\Forms(\ud_\twist\omega)=\ud_{\dual{\twist}}T_\Forms(\omega)
\]
as required.
\end{proof}
An immediate corollary of the lemma is the following.
\begin{lemma}\label{lem:tmap}
$T_\Forms$ descends to a map
\[
T_\Forms:\left(\frac{\Forms^{\twist+\bullet}(P)}{\ud_\twist\Forms^{\twist+\bullet-1}(P)}\right)^T\to\left(\frac{\Forms^{\dual{\twist}+\bullet-\dim{T}}(\dual{P})}{\ud_{\dual{\twist}}\Forms^{\dual{\twist}+\bullet-1-\dim{T}}(\dual{P})}\right)^{\dual{T}}.
\]
\end{lemma}
Now we can state the key proposition.
\begin{proposition}\label{prop:tdualforms}
\[
T_\Forms:\left(\frac{\Forms^{\twist+\bullet}(P)}{\ud_\twist\Forms^{\twist+\bullet-1}(P)}\right)^T\to\left(\frac{\Forms^{\dual{\twist}+\bullet-\dim{T}}(\dual{P})}{\ud_{\dual{\twist}}\Forms^{\dual{\twist}+\bullet-1-\dim{T}}(\dual{P})}\right)^{\dual{T}}
\]
induced by the corresponding map on differential forms is an isomorphism. Furthermore, it induces an isomorphism
\[
T_\Forms:\frac{\Forms^{\twist+\bullet}_{\im\ch}(P)}{\ud_\twist\Forms^{\twist+\bullet-1}(P)}\to\frac{\Forms^{\dual{\twist}+\bullet-\dim T}_{\im\ch}(\dual{P})}{\ud_{\dual{\twist}}\Forms^{\dual{\twist}+\bullet-1-\dim T}(\dual{P})}.
\]
\end{proposition}
\begin{proof}
It is well known \cite{BoEvMat, BoHaMat, MaWu} that 
\[
T_\Forms:\Forms^{\twist+\bullet}(P)^T\to\Forms^{\dual{\twist}+\bullet-\dim{T}}(\dual{P})^{\dual{T}}
\]
is an isomorphism. We use this to construct an explicit inverse to the map 
\[
T_\Forms:\frac{\Forms^{\twist+\bullet}_{\im\ch}(P)}{\ud_\twist\Forms^{\twist+\bullet}(P)}\to\frac{\Forms^{\dual{\twist}+\bullet-\dim T}_{\im\ch}(\dual{P})}{\ud_{\dual{\twist}}\Forms^{\dual{\twist}+\bullet-1-\dim T}(\dual{P})}.
\]
 Let $y\in\left(\Forms^{\dual{\twist}+\bullet-\dim T}(\dual{P})/\ud_{\dual{\twist}}\Forms^{\dual{\twist}+\bullet-1-\dim T}(\dual{P})\right)^T$. By Lemma \ref{lem:greg}, $y=[\omega]$, for some invariant $\omega$. We define
\[
T_\Forms^{-1}y=[T_\Forms^{-1}\bar{\omega}],
\]
where we have used the fact that $T_\Forms$ is an isomorphism on invariant forms to take the inverse of the representative. Lemma \ref{prop:tdualforms} shows that the inverse map is well defined, and it is easy to check that this map is indeed an inverse to $T_\Forms$.

To prove the remaining point, we just need to see that $T_\Forms$ maps $\Forms^{\twist+\bullet}_{\im\ch}(P)\to\Forms^{\dual{\twist}+\bullet-\dim T}_{\im\ch}(\dual{P})$. To this end, let $\omega\in\Forms^{\twist+\bullet}_{\im\ch}(P)$. Let $x\in \diffGroup{K}^{\twist+\bullet}(P)$ such that $\ch^\twist[x]=[\Curv x]=[\omega]$. By the properties of twisted differential $K$-theory (as described in App.~\ref{appendix:sec:twist})
\[
[\Curv T_{\diffGroup{K}} x]=[\Curv\pi_*\Theta_{\poincare}(\dual{\pi}^*x)]=[\dual{\pi}^*\hat{A}(P/X)]\smile\pi_*[\Curv{\Theta_{\poincare}(\dual{\pi}^*x)}],
\]
where $\Theta_{\poincare}:\diffGroup{K}^{\dual{\pi}^*\tau+\bullet}(P\times_X\dual{P})\to\diffGroup{K}^{\pi^*\dual{\tau}+\bullet}(P\times_X\dual{P})$ is defined in Th. \ref{theorem:tdualdiff}.
But $[\hat{A}(P/X)]=1$, as $P$ is a principal bundle, so that
\begin{align*}
[\Curv T_{\diffGroup{K}} x]&=\pi_*[\Curv{\Theta_{\poincare}}(\dual{\pi}^*x)]\\
&=\pi_*[\exp{\poincare}\wedge\Curv\dual{\pi}^*x]\\
&=[\pi_*\exp{\poincare}\wedge\pi^*\omega]\\
&=[T_\Forms\omega].
\end{align*}
where the second line follows from Eq.~\ref{eq:naturalityoftwists}. But then $[T_\Forms\omega]=\ch[T_{K}x]$, and thus $T_\Forms\omega\in\im\ch$, as required.
\end{proof}
Having seen how the $T$-map acts on differential forms, we now examine the $T$-map reduced to twisted $K$-theory.

\begin{proposition}\label{prop:tdual}
The sequence
\[
T_K:\xymatrix{
K^{\twist+\bullet}(P)\ar[r]^<<<<<{\dual{\pi}^*}&K^{{\hat{\pi}^{*}}\twist+\bullet}({P\times_X\dual{P}})\ar[r]&K^{{{\pi}^{*}}\dual{\twist}+\bullet}({P\times_X\dual{P}})\ar[r]^<<<<<{\pi_*}&K^{\dual{\twist}+\bullet-\dim T}(\dual{P})
}
\]
is an isomorphism, where this sequence (and in particular the twists and the isomorphism between them) is obtained by ``forgetting the geometry"\footnote{The passage from twisted differential $K$-theory to $K$-theory is explained in Appendix \ref{appendix:sec:twist}.} of the sequence Th.~\ref{theorem:tdualdiff}.
\end{proposition}
As we discuss in the next section, Prop.~\ref{prop:tdual} is essentially the topological $T$-duality found in \cite{BuSchi, BuRuSchi, BoHaMat, BoEvMat}. In particular, $((P,\tau),(\dual{P},\dual{\tau}),\Theta_\poincare)$ is a $T$-duality triple in the sense of \cite{BuRuSchi}, and the proposition follows from their theorem. 
We may finally prove Th.~\ref{theorem:tdualdiff}.
\begin{proof}[Proof of Th.~\ref{theorem:tdualdiff}]
We will use the five lemma on Fig.~\ref{diagram:five}.
\begin{sidewaysfigure}[!]
\[
\xymatrix{
0\ar[r]\ar[d]&\frac{\Forms^{\twist+\bullet-1}_{\im\ch}(P)}{\ud_\twist\Forms^{\twist+\bullet}(P)}\ar[r]\ar[d]^{T_\Forms}&\left(\frac{\Forms^{\twist+\bullet-1}(P)}{\ud_\twist\Forms^{\twist+\bullet}(P)}\right)^T\ar[r]^{\formsmap}\ar[d]^{T_\Forms}&
\diffGroup{K}^{\diffEl{\twist}+\bullet}(P)^T\ar[r]^\kmap\ar[d]^{T_{\diffGroup{K}}}&
K^{\twist+\bullet}(P)\ar[r]\ar[d]^{T_K}&
0\ar[d]\\
0\ar[r]&\frac{\Forms^{\dual{\twist}+\bullet-1-\dim T}_{\im\ch}(\dual{P})}{\ud_\twist\Forms^{\dual{\twist}+\bullet-\dim T}(\dual{P})} \ar[r]&\left(\frac{\Forms^{\dual{\twist}+\bullet-1-\dim{T}}(\dual{P})}{\ud_{\dual{\twist}}\Forms^{\dual{\twist}+\bullet-\dim{T}}(\dual{P})}\right)^{\dual{T}}
\ar[r]^{\formsmap}&
\diffGroup{K}^{\diffEl{\dual{\twist}}+\bullet-\dim{T}}(\dual{P})^{\dual{T}}\ar[r]^\kmap&
K^{\dual{\twist}+\bullet-\dim{T}}(\dual{P})\ar[r]&
0
}
\]
\caption{Diagram used to prove $T$-duality for differential $K$-theory}
\label{diagram:five}
\end{sidewaysfigure}
Prop.~\ref{prop:sequence} shows the top and bottom rows are exact. The diagram commutes by naturality. Th.~\ref{prop:tdual} and Prop.~\ref{prop:tdualforms} then show that the required maps down are isomorphisms, proving the theorem.
\end{proof}

\newpage
\subsection{Relation to other approaches}\label{subsec:relationtoother}
In this section we show our approach to $T$-duality fits with the various points of view on \emph{topological} $T$-duality found in  \cite{BoEvMat, BoHaMat, BuRuSchi, BuSchi}. We begin by considering the very simplest situation and take the base $X$ to be a point. In this case we have the picture
\[
\xymatrix{
&T\times \dual{T}\ar[dl]_{\dual{\pi}}\ar[dr]^{\pi}&\\
T\ar[dr]_\pi&&\dual{T}\ar[dl]^{\dual{\pi}}\\
&\pt&
}
\]
where the maps are the canonical projections, and $T\to\{*\}$ is obtained from the natural pointed map $\{*\}\to BT$ (and similarly for $\dual{T}$). We note that $\Mor{\diffEl{H}^4(\pt;\Z)}$ is trivial, hence there is no choice of $\sigma$. Furthermore, $\diffEl{P}$ and $\diffEl{\dual{P}}$ are both canonically the trivial objects in $\diffEl{H}^2(\pt;\Lattice)$ and $\diffEl{H}^2(\pt;\dual{\Lattice})$ respectively. Pulling these objects back to $T$, $\dual{T}$ respectively, we see that the morphisms $\delta_{\diffEl{P}}$, $\delta_{\dual{\diffEl{P}}}$ are automorphisms of $0\in \diffEl{H}^2(T;\Lattice)$, $0\in \diffEl{H}^2(\dual{T};\dual{\Lattice})$, and are thus given by objects in $\diffEl{H}^1(T;\Lattice)$ and $\diffEl{H}^1(\dual{T};\dual{\Lattice})$. We recall that classes in the groupoid\footnote{See Appendix \ref{appendix:sec:geometric} for a discussion of the various groupoids  associated to differential cohomology.}  $\diffTop{H}^1(-;\Lattice)$ are maps from the space to $T$, and working through the definitions, we see that the class of $\delta_{\diffEl{P}}$ represents the identity map $\mathrm{id}_T:T\to T$ (and similarly, $\dual{\delta}_{\dual{\diffEl{P}}}$ is in the class of the map $\mathrm{id}_{\dual{T}}:\dual{T}\to\dual{T}$). Identifying the classes in the differential cohomology groups that $\delta_{\diffEl{P}}$ and its dual lie in suffices for our purposes, but one may easily follow through the discussion on trivialisations in Appendix \ref{appendix:sec:actionontwist} to identify actual objects $\delta_{\diffEl{P}}$, $\dual{\delta}_{\dual{\diffEl{P}}}$. We note that $\Curv{\delta_{\diffEl{P}}}=\theta_T$, the Maurer-Cartan form on $T$, and $\Curv{\dual{\delta}_{\diffEl{\dual{P}}}}=\theta_{\dual{T}}$.

The twists
\begin{align*}
\twist&=\sigma-\delta_{\diffEl{P}}\cdot\dual{\diffEl{P}},\\
\dual{\twist}&=\sigma-\diffEl{P}\cdot\delta_{\dual{\diffEl{P}}},
\end{align*}
are both canonically trivial. Nonetheless, we will see that the morphism 
\[
\delta_{\diffEl{P}}\cdot\dual{\delta}_{\dual{\diffEl{P}}}:\twist\to\dual{\twist}.
\]
is non-trivial.
As an automorphism of $0\in\diffEl{H}^3(T\times\dual{T};\Z)$, it is canonically an object in $\diffEl{H}^2(T\times\dual{T};\Z)$. Such objects determine \emph{line-bundles with connection}\footnote{For a discussion of precisely how this occurs, see Appendix \ref{appendix:sec:gerbe}.}. Chasing through the definitions, and recalling that $\delta_{\diffEl{P}}$ and its dual are essentially the identity maps on the respective tori, we see that the line bundle $\delta_{\diffEl{P}}\cdot\dual{\delta}_{\dual{\diffEl{P}}}$ is  canonically isomorphic to the \emph{Poincar\'e line bundle} with connection. For $T=S^1$, this is the line bundle with connection described in in Appendix \ref{appendix:sec:tinv}. 

As the twists $\twist$ and $\dual{\twist}$ are both trivial, we have that $\delta_{\diffEl{P}}\cdot\dual{\delta}_{\dual{\diffEl{P}}}$ induces an automorphism of $\diffGroup{K}^{\bullet}(T\times \dual{T})$. Concretely, this is the automorphism induced by tensoring by the Poincar\'e line bundle, seen as an element of $\diffGroup{K}^\bullet(T\times \dual{T})$. 

At the level of forms (i.e. restricting to topologically trivial elements of $\diffGroup{K}$), the above discussion shows that 
\begin{equation}\label{eq:formstmap}
T_\Forms:\omega\mapsto\int_T\exp(\theta_T\cdot\theta_{\dual{T}})\wedge\omega,
\end{equation}
where $\omega\in\Forms^\bullet(T)$, and $\theta_T\cdot\theta_{\dual{T}}$ is the differential form obtained by the sequence
\[
\xymatrix{
\Forms^n(X;V)\otimes\Forms^m(X;\dual{V})\ar[r]^<<<<<\bigwedge&\Forms^{n+m}(X;V\otimes\dual{V})\ar[r]&\Forms^{n+m}(X;\R).
}
\]
with the second arrow being induced by the pairing $V\otimes\dual{V}\to\R$.
Eq.~\ref{eq:formstmap} is expression for the $T$-duality isomorphism found in the physics literature \cite{BoHaMat,BoEvMat}.

We now return to the general case where we have geometric $T$ and $\dual{T}$ bundles $(P,\nabla)$, $(\dual{P},\dual{\nabla})\to X$, with $X$ a smooth manifold, and a trivialisation $\sigma:0\to\diffEl{P}\cdot\dual{\diffEl{P}}$. As explained in Sec.~\ref{sec:difftwist}, this data allows us to construct canonical objects $\twist\in\diffEl{H}^3(P;\Z)$, $\dual{\twist}\in\diffEl{H}^3(\dual{P};\Z)$ explicitly given by
\begin{align*}
\twist&=\pi^*\sigma-\delta_{\diffEl{P}}\cdot\pi^*\dual{\diffEl{P}},\\
\dual{\twist}&=\dual{\pi}^*\sigma-\dual{\pi}^*\diffEl{P}\cdot\dual{\delta}_{\diffEl{\dual{P}}}.
\end{align*}
It is easy to see that, when $\dim{T}=1$,
\begin{equation}\label{eq:geomt}
\pi_*\twist\cong\diffEl{\dual{P}},\quad\dual{\pi}_*\dual{\twist}\cong\diffEl{P},
\end{equation}
where by ``$\cong$'' we mean equality as classes in differential cohomology (ie, isomorphic in the \emph{topological} category). Forgetting the geometry, Eq.~\ref{eq:geomt} reduces to the criterion of $T$-duality described in Sec.~\ref{subsec:tdualtop} (specifically Eq.~\ref{eq:introphcriterion}).  

The objects $\twist$ and $\dual{\twist}$ in turn give rise to twists of differential $K$-theory, and we construct the canonical homomorphism defined in Th. \ref{theorem:tdualdiff}
which, on forgetting the geometry, induces a homomorphism
\[
T_K:K^{\twist+\bullet}(P)\to K^{\dual{\twist}+\bullet-\dim T}(\dual{P}).
\]
The triple $((P,\twist),(\dual{P},\dual{\twist}),\Theta_\poincare)$ is precisely a $T$-duality triple in the sense of \cite{BuRuSchi}. In particular, we note that pulling back along any map $i:\{*\}\to X$ one obtains the situation described at the beginning of this section, so that over a fibre the isomorphism of twists is induced by multiplying by the Poincar\'e line bundle.
\newpage
\appendix

\section{Differential cohomology}\label{appendix:differential}
This appendix discusses various definitions and formal properties in differential cohomology theory. Section \ref{app:lattices} describes a model for differential cohomology with values in a lattice, following Hopkins and Singer \cite{HoSi}. Section \ref{appendix:sec:geometric} discusses two groupoids naturally associated to differential cohomology, and in particular discusses the notion of geometric trivialisation so important to our formulation of $T$-duality for differential $K$-theory. Section \ref{appendix:sec:twist} describes the formal properties of twisted differential $K$-theory that we use in this paper. Finally, Section \ref{appendix:sec:gerbe} constructs a canonical twist of differential $K$-theory from an object in $\diffEl{H}^3(-)$.
\subsection{Differential cohomology with values in lattices}\label{app:lattices}
This section contains a brief description of a model for differential  cohomology with values in a lattice contained in a vector space. This is a straightforward generalisation of the Hopkins-Singer model used in \cite{HoSi} for ordinary singular cohomology, and is presented here to settle notation.

Let $V$ be a finite dimensional real vector space, and $\Lambda\subseteq{V}$ a full lattice.\footnote{A lattice $\Lambda\subseteq{V}$ is \emph{full} if $\Lambda\otimes_{\Z}\R\simeq{V}$.} Let $X$ be a compact manifold. Denote with $C^{\bullet}(X;\Lambda)$ and $C^{\bullet}(X;V)$ the smooth singular cochain complexes with values in $\Lambda$ and $V$, respectively, and with $H^{\bullet}(X;\Lambda)$ and $H^{\bullet}(X;V)$ the associated cohomology groups. Finally, denote with $\Forms^{\bullet}(X;V)$ the space of differential forms with values in $V$, namely $\Forms^{\bullet}(X;V):=\Forms^{\bullet}(X)\otimes_{\mathbb{R}}V$. For any $V$-valued form, we will denote with $\tilde{\omega}$ the $V$-valued singular cocyle obtained by integration, i.e. for $\omega=\eta\otimes{v}$ we have
\[
\tilde{\omega}(\sigma):=\left(\int_{\sigma}\eta\right){v}
\]
for any $\sigma\in{C_{\bullet}(X)}$.

Define the bi-graded abelian group
\[
\mathcal{C}^{p}(q)(X;\Lambda):=\begin{cases}
C^{p}(X;\Lambda)\times C^{p-1}(X;V)\times \Forms^{p}(X;V),&\quad p\geq{q}\\
C^{p}(X;\Lambda)\times C^{p-1}(X;V),&\quad p<{q}
\end{cases} 
\]
and consider for any $q$ the differential $\check{d}:\mathcal{C}^{p}(q)(X)\to\mathcal{C}^{p+1}(q)(X)$ defined as
\[
\check{d}(c,h,\omega):=(\delta{c},c_{V}-\tilde{\omega}-\delta{h},\ud\omega),\quad{p\geq q}
\]
and
\[
\check{d}(c,h):=\begin{cases}
(\delta{c},c_{V}-\delta{h},0),&{p=q-1}\\
(\delta{c},c_{V}-\delta{h}),&{p<q-1}
\end{cases}
\]
where $c_{V}$ denotes the image of $c$ under the map induced by the inclusion $\Lambda\subseteq{V}$. We denote the cohomology of the complex $\left(\mathcal{C}^{\bullet}(q)(X),\check{d}\right)$ by $\diffGroup{H}^{\bullet}(q)(X;\Lambda)$. Define 
\[
\diffGroup{H}^{p}(X;\Lambda):=\diffGroup{H}^{p}(p)(X;\Lambda),
\] 
the \emph{$p$'th differential cohomology group of $X$ with values in $\Lambda$}. For any $p$, the groups $\diffGroup{H}^{p}(X;\Lambda)$ fit in the following exact sequences
\begin{gather}\label{eq:diffexactcurve}
0\to{H^{p-1}(X;V/\Lambda)}\to\diffGroup{H}^{p}(X;\Lambda)\xrightarrow{\Curv}\Forms_{\Lambda}^{p}(X;V)\to{0},\\
\label{eq:diffexactdelta}
0\to\frac{\Forms^{p-1}(X;V)}{\Forms_{\Lambda}^{p-1}(X;V)}\xrightarrow{i}\diffGroup{H}^{p}(X;\Lambda)\xrightarrow{\delta}{H}^{p}(X;\Lambda)\to{0}.
\end{gather}
There is a \emph{cup product} $\smile$ refining the cup product for cohomology with coefficients. More precisely, for $\Lambda\subseteq{V}$ and $\Lambda'\subseteq{V'}$ we have a bilinear map 
\[
\smile:\mathcal{C}^{p}(q)(X;\Lambda)\times\mathcal{C}^{p'}(q')(X;\Lambda')\to\mathcal{C}^{p+p'}(q+q')(X;\Lambda\otimes_{\mathbb{Z}}\Lambda')
\]
which is compatible with the differential $\check{d}$, and moreover satisfies
\begin{align*}
\Curv(x\smile{y})&=\Curv(x)\wedge\Curv(y),\\
\delta(x\smile{y})&=\delta(x)\smile\delta(y).
\end{align*}
See \cite{HoSi} for details on the construction of $\smile$.

When $V'=\Hom{V}{\mathbb{R}}$ and $\Lambda'=\Hom{\Lambda}{\mathbb{Z}}$ there is a pairing 
\[
\cdot:\mathcal{C}^{p}(q)(X;\Lambda)\times\mathcal{C}^{p'}(q')(X;\Lambda')\to\mathcal{C}^{p+p'}(q+q')(X;\mathbb{Z})
\] 
defined by the composition of the cup-product with the natural pairing  $\Lambda\otimes\Lambda'\to\Z$. 

The group of smooth maps from $X$ to $V/\Lambda$ is isomorphic to the group $\diffGroup{H}^1(X;\Lambda)$. The group of isomorphism classes of principal $V/\Lambda$-bundles over $X$ with connection is isomorphic to  $\diffGroup{H}^2(X;\Lambda)$. In fact, one may associate to a \emph{cocycle} in $\diffEl{C}^2(2)(X;\Lambda)$ a canonical principal $V/\Lambda$-bundle with connection over $X$, in a manner analogous to the construction of a principal $U(1)$-bundle with connection associated to a cocycle in $\diffEl{C}^2(2)(X;\Z)$ described in Sec.~\ref{appendix:sec:gerbe}.
\subsection{Geometric trivialisation}\label{appendix:sec:geometric}
In this section we describe two groupoids associated with the differential cohomology group {$\diffGroup{H}^{p}(X)$}: the first has as its connected components the differential cohomology group, and the second, which we call the \emph{geometric} groupoid, is the one we use extensively in Section \ref{sec:smoothtdual}. Both of these groupoids are discussed in \cite{Fr:Dirac, HoSi}. In fact, both of these groupoids associated to $\diffGroup{H}^{p}(X)$ may be extended to be $p$-groupoids: we will not formalise higher structure, but it should be clear how to extend the construction recursively.

Let $X$ be a compact manifold, and denote with $\diffEl{Z}^{p}(X)$ the kernel of the differential $\check{d}:\mathcal{C}^{p}(p)(X)\to\mathcal{C}^{p+1}(p)(X)$. We define the category $\diffTop{H}^{p}(X)$ as having the following objects and morphisms 
\begin{align*}
\Obj{\diffTop{H}^{p}(X)}&:=\diffEl{Z}^{p}(X),\\
{\rm Mor}(x,y)&:=\left\{\alpha\in\mathcal{C}^{p-1}(p)(X):x=y+\check{d}\alpha\right\}.
\end{align*}
The category $\diffTop{H}^{p}(X)$ is a groupoid, since all the morphisms are invertible, and its connected components are precisely the elements of the differential cohomology group $\diffGroup{H}^{p}(X)$. The objects of $\diffTop{H}^{p}(X)$ are endowed with a structure of an abelian group, in particular it has a distinguished 0 object. We will say that an object $x$ of $\diffTop{H}^{p}(X)$ is \emph{trivialisable} if ${\rm Mor}(x,0)$ is not empty, and refer to elements $\alpha$ in ${\rm Mor}(x,0)$ as \emph{trivialisations}. In other words, we have $x=\check{d}\alpha$ if $x$ is trivialisable. One can see that the class of a trivialisable element $x$ in the group $\diffGroup{H}^{p}(X)$ satisfies
\begin{align*}
\Curv(x)&=0,\\
\delta(x)&=0,
\end{align*}
where $\Curv$ and $\delta$ are the maps from the sequences in Eqs \ref{eq:diffexactcurve} and \ref{eq:diffexactdelta}.

We now define the \emph{geometric} groupoid associated to $\diffGroup{H}^p(X)$, which we denote by $\diffEl{H}^p(X)$. The resultant notion of trivialisation is the one we use in Sec.~\ref{sec:smoothtdual}. The objects of $\diffEl{H}^p(X)$ are again  the differential co-cycles $\diffEl{Z}^p(X)$. A geometric morphism between two co-cycles $x$ and $y$ is an $\alpha\in\diffEl{C}^{p-1}(p-1)(X)$ such that $x=y+\check{d}\alpha$, where we regard $\check{d}\alpha\in\diffEl{C}^{p}(p)(X)$ via the inclusion $\diffEl{C}^{p+1}(p)(X)\subset\diffEl{C}^{p+1}(p+1)(X)$. A \emph{geometric trivialisation} of an object $x$ is a geometric morphism between $0$ and $x$. Any object for which such a geometric trivialisation exists is called \emph{geometrically trivialisable}. Notice that the image of a geometrically trivialisable object $x$ in {$\diffGroup{H}^{p}(X)$} is in general \emph{non-zero}, since $\alpha$ is not in general an element in $\diffEl{C}^{p-1}(p)(X)$. The geometric groupoid is the groupoid of relevance in this paper. The cup-product extends to morphisms (for both groupoids):
\[
\smile:\Obj{\diffEl{H}^p(X)}\times\Mor{\diffEl{H}^q(X)}\to\Mor{\diffEl{H}^{p+q}(X)}.
\]

We give a short motivating example for the adjective ``geometric''. Let $\mathcal{L}\to{X}$ be a $U(1)$-principal bundle with connection $\nabla$, and suppose that $\mathcal{L}\to{X}$ is topologically trivial. In particular it represents the zero class in $H^{2}(X;\mathbb{Z})$. Let $c$ be a cocycle representative of $\mathcal{L}$: then $c$ is exact. Any section $s:X\to\mathcal{L}$ allows to construct an explicit  trivialisation of the cohomology class, namely a cocycle $\alpha$ such that $c=\delta{\alpha}$. In particular, any section $s$ give an isomorphism $\mathcal{L}\simeq{X}\times{\mathbb{C}}$. Nevertheless, the section $s$ may not induce a connection preserving isomorphism with $(X\times\C,\ud)\to X$, and thus not trivialise $(\mathcal{L},\nabla)$ in $\diffTop{H}^2(X)$. In the language of cochains, the cocyle $\check{c}\in\diffEl{Z}^{2}(2)(X)$ representing $(\mathcal{L},\nabla)$ will in general \emph{not} be exact, hence it will not represent the zero class in $\diffGroup{H}^{2}(X)$. However, we explain in Appendix \ref{appendix:sec:actionontwist} how the section $s$ induces a geometric trivialisation, i.e. a morphism $0\to(\mathcal{L},\nabla)$ in $\diffEl{H}^2(X)$.
\subsection{Twisted differential $K$-theory}\label{appendix:sec:twist}
In this section we describe the formal properties of \emph{twisted} differential $K$-theory we require for our work. The study of twisted differential $K$-theory is still young, and some of these properties have the character of ``folk theorems''. However, Carey et al.~\cite{CaMiWa} construct a model of twisted differential $K$-theory with many of the properties we require. We will in future work give a construction of twisted differential $K$-theory adapted to our language.

Our basic model for twisted $K$-theory is that described in \cite{FHT1}. 
 
Let $X$ be a compact manifold. We may form the trivial groupoid\footnote{For us, groupoids are {smooth}, and assumed to be \emph{local quotient groupoids}: that is, they admit a countable open cover by sub-groupoids, each of which is weakly equivalent to a compact Lie group acting on a 
Hausdorff space. We denote a groupoid by a tuple $(X_0,X_1)$, where $X_0$ is the object space, and $X_1$ is the morphism space. Notational conventions are as in \cite{FHT1}, which also contains an excellent discussion of the basic notions surrounding groupoids.} $\mathcal{X}=(X,X)$. The groupoid of twists of $\diffGroup{K}^\bullet(X)$, $\gTwist(X)$,\footnote{As explained in \cite{FHT1}, $\gTwist(X)$ is, in fact, a 2-groupoid.} is the groupoid of \emph{geometric central extensions} $(\mathcal{P},(L,\nabla),\omega)$ with $P$ locally equivalent to $\mathcal{X}$.\footnote{As of writing, no construction of twisted differential $K$-theory has been done using these twists, but it is reasonable to expect that this groupoid twists differential $K$-theory. Indeed, one may see precisely these objects twisting differential $K$-theory in slides by Dan Freed ``Dirac Charge Quantization, $K$-Theory, and Orientifolds" found at \url{http://www.ma.utexas.edu/users/dafr/paris_nt.pdf}.} By a geometric central extension of a groupoid, we mean the following.
\begin{definition}[Geometric central extension]
A geometric central extension of a groupoid $\mathcal{P}=(P_0,P_1)$ is a tuple $(\mathcal{P},(L,\nabla),\omega)$, where $L\to P_1$ is a principal $U(1)$-bundle with connection $\nabla$, and $\omega\in\Forms^2(P_0)$, satisfying the following conditions:
\begin{enumerate}
\item $L\to P_1$ is a central extension of groupoids,
\item $(L,\nabla)$ satisfies the commutative diagrams in Def.~2.4 of \cite{FHT1} as a line bundle with connection,
\item $p^*_1\omega-p^*_0\omega=\frac{i}{2\pi}\Omega^{\nabla}$.
\end{enumerate}
\end{definition}
We note that bundle gerbes with curving and connective structure, and $PU(H)$-bundles with connection and $B$-field are both essentially objects in $\gTwist(X)$. In Sec.~\ref{appendix:sec:gerbe} we describe a functor from $\diffEl{H}^3(X)$ to $\gTwist(X)$.

There are functors
\begin{align}
\label{forget}F&:\gTwist(X)\to\Twist(X),\\
\label{curv}\Curv&:\gTwist(X)\to\Forms^3_{\ud=0}(X).
\end{align}
The first, ``forgetting the geometry'', constructs a twist of $K$-theory in the sense of \cite{FHT1}\footnote{The twists in \cite{FHT1} have an additional grading. All our twists are taken to be even in their grading.} by 
\[
(\mathcal{P},(L,\nabla),\omega)\mapsto(\mathcal{P},L).
\]
The second functor, ``curvature'', gives an object in the category\footnote{Objects in $\Forms^3_{\ud=0}(X)$ are closed 3-forms on $X$, and a morphism $\alpha:\omega\to\omega'$ is an $\alpha\in\Forms^2(X)/{\ud\Forms^1(X)}$ such that $\omega'=\omega+\ud\alpha$.} $\Forms^3_{\ud=0}(X)$. These are twists of the complex of differential forms: from $\omega\in\Forms^3_{\ud=0}(X)$ one forms the $\Z/2\Z$ graded complex
\[
(\Forms^{\omega+\bullet}(X),\ud_\omega):=(\Forms^\bullet(X),\ud+\omega).
\]
The twisted cohomology group $H^{\omega+\bullet}(X)$ is defined to be the cohomology of the complex $(\Forms^{\omega+\bullet}(X),\ud_\omega)$. When the twist is induced by an object $\tau\in\diffEl{H}^3(X)$, the functor $\Curv$ is induced by the ``$\Curv$'' morphism in Eq.~\ref{eq:diffexactcurve}. 

An object $\twist\in\gTwist$ allows one to define a Chern character $\ch^\tau:K^{\twist+\bullet}(X)\to H^{\tau+\bullet}(X)$.

 To an \emph{object} $\twist\in\gTwist(-)$, one may associate an abelian group ${\mathcal{K}^{\twist+\bullet}}(-)$ satisfying the following properties:
\begin{itemize}
\item[-]\emph{functoriality}: for any smooth map $f:X\to{Y}$, we have a pullback morphism
\[
f^{*}:\diffGroup{K}^{\twist+\bullet}(Y)\to\mathcal{K}^{f^{*}\twist+\bullet}(X),
\]
\item[-]\emph{exact sequences}: the following natural exact sequence holds
\begin{gather}
\label{eq:diffkexact}\xymatrix{
0\ar[r]&\left(\frac{\Forms^{\twist+\bullet-1}(X)}{\Forms^{\twist+\bullet-1}_{\im\ch}(X)}\right)\ar[r]^\formsmap&\diffGroup{K}^{\twist+\bullet}(X)\ar[r]^\kmap&
K^{\twist+\bullet}(X)\ar[r]&
0
},
\\
\label{eq:diffkexactcurv}\xymatrix{
0\ar[r]&K^{\twist+\bullet-1}(X,\mathbb{R}/\mathbb{Z})\ar[r]&\diffGroup{K}^{\twist+\bullet}(X)\ar[r]^{\Curv}&\Forms_{K}^{\twist+\bullet}(X)\ar[r]&0,
}
\end{gather}
where the twist of $K$-theory and of the differential forms are obtained from the functors $F$ and $\Curv$ in Eqs \ref{forget} and \ref{curv}. The composition of $\Curv$ and the map $\Omega^\tau_{\ud_\tau=0}(X)\to H^{\tau+\bullet}(X)$ is the twisted Chern character map $\ch^\tau$ mentioned above.
\item[-]\emph{pushforward}: for any differential $K$-oriented Riemannian family $f:X\to Y$ there exists a ``wrong way'' map
\[
f_{*}:\diffGroup{K}^{f^{*}\twist}(X)\to\diffGroup{K}^{\twist}(Y).
\]
In particular, we assume the following property of the pushforward: 
\[
f_{*}(i([\omega]))=i\left(\left[\int_{X/Y}{\rm Td}(\nabla^{X/Y})\wedge\omega\right]\right),
\]
for all $\omega\in\Forms^{f^{*}\twist+\bullet-1}(X)$, where ${\rm Td}(\nabla^{X/Y})$ denotes the Todd form of the Levi-Civita connection $\nabla^{X/Y}$ associated to the Riemannian map $f$.
\item[-]\emph{naturality of twists}: for any morphism $\alpha\in{\rm Mor}(\twist,\twist')$ there is a natural isomorphism
\[
\phi_{\alpha}:\mathcal{K}^{\twist+\bullet}(X)\xrightarrow{\simeq}\mathcal{K}^{\twist'+\bullet}(X).
\]
The isomorphism $\phi_{\alpha}$ satisfies
\begin{align}\label{eq:naturalityoftwists}
\phi_{\alpha}\left(i\left([\omega]\right)\right)&=i\left(\left[{e}^{\Curv{\alpha}}\wedge\omega\right]\right),\\
\Curv\left({\phi_{\alpha}x}\right)&={e}^{\Curv{\alpha}}\wedge\Curv{x},
\end{align}
for all $\omega\in\Forms^{\twist+\bullet-1}(X)$, and $x\in\diffGroup{K}^{\twist+\bullet}(X)$. Moreover, if $\alpha'\in{\rm Mor}(\twist,\twist')$ is such that there exists $\beta\in{\rm Mor}(\alpha,\alpha')$, then $\phi_{\alpha'}=\phi_{\alpha}$.
\end{itemize}
\subsection{Twists of differential $K$-theory from Hopkins-Singer cocycles}\label{appendix:sec:gerbe}
In this section we sketch the construction of a  functor $\diffEl{H}^3(X)\to\gTwist(X)$. Before constructing the functor, we briefly recall a construction ``one degree lower'', found in Hopkins and Singer \cite{HoSi}, which associates canonical circle bundles with connections to objects in $\diffEl{H}^2(X)$.

One presentation of a circle bundle $P\to X$ with connection $\nabla$ is as an assignment to each open set $U\subset{X}$ a (possibly empty) principal homogeneous space $\Gamma(U)$ for the group $C^\infty(U,U(1))$. Intuitively, $\Gamma(U)$ is the space of local sections of $P$ over $U$. A connection is then given by an assignment   
\[
\nabla:\Gamma(U)\to\Forms^{1}(U)
\]
with the ``equivariance" condition 
\[
\nabla(g\cdot{s})=\nabla(s)+g^{-1}dg
\] 
for any $s\in\Gamma(U)$ and $g\in C^\infty(U,U(1))$: $\nabla(s)$ gives the 1-form representative of the connection $\nabla$ in the trivialisation induced by $s$. Such an assignment $\Gamma$ should ``glue'' properly: it should be a sheaf of torsors over the sheaf of circle valued functions over $X$. 

We now present a line bundle from a given $x\in\diffEl{H}^2(X)$. For each open set $U\subset{X}$, define
\begin{equation}
\Gamma(U):=\left\{s\in\mathcal{C}^{1}(1)(X):x|_{U}=\check{d}s\right\}/\sim
\end{equation}
where $s'\sim{s}$ iff there exists $t\in\mathcal{C}^{0}(1)(X)$ such that $s'=s+\check{d}t$. In other words, we assign to $U$ the space of geometric trivialisations of $x|_{U}$ up to boundaries. For any $s,s'\in\Gamma(U)$, $\alpha:=s-s'$ is an element in $\mathcal{Z}^{1}(1)(X)$ up to boundaries. It is thus  a class in $\diffGroup{H}^{1}(X)$, i.e. an $\R/\Z$ valued function on $U$. Identifying $\mathbb{R}/\mathbb{Z}$ with $U(1)$ gives the required torsor over the sheaf of $U(1)$-valued functions. The connection $\nabla$ can be defined by
\begin{equation}
\nabla(s):=\Curv(s)
\end{equation}
$\forall s\in\Gamma(U)$, and the equivariance condition follows.

We now extend this construction to produce the desired functor $\diffEl{H}^3(X)\to\gTwist(X)$. Let $x\in\diffEl{H}^{3}(X)$. Define
\begin{equation}
\Gamma^{B}(U):=\left\{s\in\mathcal{C}^{2}(2)(X):ds=x|_{U}\right\}
\end{equation}
By the comment above, for each $U$ we have that $\Gamma^{B}(U)$ is a torsor over $\mathcal{H}^{2}(U)$. Define 
\[
P_0=\coprod_{\substack{
U\text{ open in }X\\
s\in\Gamma(U)
}}U_s.
\]
In other words, a point of $P_0$ is a point \emph{in} an open set $U\subset X$, ``coloured" by an $s\in\Gamma(U)$.
The space $P_1\equiv P_0\times_X P_0$ may be identified with 
\[
P^{[2]}=\coprod_{\substack{
U,V\text{ open in }X\\
s\in\Gamma(U),\,t\in\Gamma(V)
}}U_s\cap V_t,
\]
where here the intersection is \emph{ordered}. It is easy to see there is a local equivalence of groupoids between $\mathcal{P}=(P_0,P_1)$ and canonical groupoid $\mathcal{X}=(X,X)$. We now build a canonical central extension of $\mathcal{P}$. We do more: we construct a principal $U(1)$-bundle with connection $(\mathcal{L},\nabla)\to P_1$: to each $U_s\cap U'_{s'}$ we assign the the $U(1)$-bundle with connection $(\mathcal{L}_{s,s'},\nabla_{s-s'})$ associated to the Hopkins-Singer 2-cocyle $s-s'$, as described above. Taken together these give the desired $U(1)$ bundle. 
On $P^{[3]}=U_s\cap U'_{s'}\cap U''_{s''}$, we have that the equation 
\begin{equation}
(s-s')+(s'-s'')=(s-s'')
\end{equation}
induces the canonical isomorphism
\begin{displaymath}
\mathcal{L}_{ss'}\otimes\mathcal{L}_{s's''}\simeq{\mathcal{L}_{ss''}}
\end{displaymath}
of $U(1)$-bundles with connections. Similar considerations show that all the required commutative diagrams are satisfied (as $U(1)$-bundles \emph{with connection}), and thus $(\mathcal{P},L)$ is a central extension of the groupoid $\mathcal{P}$.

We have already seen that the line bundle $\mathcal{L}$ has a connection, which also satisfies the correct compatibility on triple overlaps. We now construct the two-form $\omega\in\Forms^2(P_0)$ required to make $(\mathcal{P},(L,\nabla))$ into a geometric central extension of $\mathcal{P}$. The ``$B$-field'' is defined by assigning to each $U_s$ the 2-form $\Curv(s)\in\Forms^{2}(U_s)$. The assignment is easily seen to satisfy the required compatibility condition:
\begin{displaymath}
\Curv(s)-\Curv(s')=\frac{i}{2\pi}\Omega^\nabla_{ss'}.
\end{displaymath}

It remains to show how geometric morphisms $\sigma:x\to y$ (objects in $\alpha\in\mathcal{C}^2(2)(X)$ such that $x=y+d\alpha$) induce morphisms of twists.  The element $\alpha$ naturally induces a map on torsors associated to $x$ and $y$, $\alpha:\Gamma_x(U)\to\Gamma_y(U)$ by sending $s$ to $s+\alpha$. This in turn, induces the required morphism.
\section{Notes on the $T$-action}\label{appendix:groupactions}
The action of the torus on the twisted $K$-theory and differential twisted $K$-theory of the bundles appearing in $T$-duality is in itself an interesting topic. In this appendix, we explain some basic results -- in particular, we show how the twists arising in $T$-duality are acted on by the torus. We also exhibit an explicit class in $K^0(T^2)$ that has \emph{no} fixed pre-image in $\diffGroup{K}^0(T^2)$ (although it does have a ``geometrically invariant'' pre-image). This suggests that restricting attention to the action of $T$-duality on the fixed subgroup of differential $K$-theory is far too restrictive, and further motivates our choice of the ``geometrically invariant'' subgroup as the correct subgroup suitable for $T$-duality.
\subsection{The action of $T$ and $\dual{T}$ on the twistings}\label{appendix:sec:actionontwist}
The goal of this section is to examine the action of the torus and its dual on the twistings involved in $T$-duality. To this end, we recall that the bundles $(P,\nabla)$ and $(\dual{P},\dual{\nabla})$ come with classifying maps adapted to the connections $f:X\to BT$, $\bar{f}:X\to ET$, and $\dual{f}:X\to B\dual{T}$, $\bar{\dual{f}}:X\to E\dual{T}$ so that the diagram below (and its dual)
\[
\xymatrix{
(P,\nabla)\ar[r]^{\bar{f}}\ar[d]^\pi&(ET,\nabla^u)\ar[d]^{\pi^u}\\
X\ar[r]^{f}&BT
}
\]
commutes.

We now fix\footnote{Everything done henceforth follows through word-for-word for the bundle $\dual{P}$.} $c\in \diffEl{H}^2(BT;\Lattice)$ representing the universal $U(1)$-bundle with connection $\pi^u:(ET,\nabla^u)\to BT$. The element $(\pi^u)^*c\in \diffEl{H}^2(ET;\Lattice)$ is trivial, and we fix a trivialisation in $\Mor{\diffEl{H}^2(ET;\Lattice)}$
\[
\alpha:0\to (\pi^u)^*c
\]
with $\Curv \alpha$ $T$-invariant.
With this understood, we see that the element $\diffEl{P}\in H^2(X;\Lattice)$ is given by
\[
\diffEl{P}=f^*c.
\]
and thus 
\[
\bar{f}^*\alpha:0\to\pi^*\diffEl{P}.
\]
We note that any global section $s:X\to P$ induces a geometric trivialisation $t_s:0\to\diffEl{P}$ defined by
\[
t_s=s^*\bar{f}^*\alpha.
\]

We now consider the pullback bundle $p:\pi^*P\to P$ equipped with the pullback connection. The map $\pi^*f:P\to BT$ classifies $\pi^*P$, and there is a natural map $p^*\bar{f}:\pi^*P\to ET$ such that $\pi^u\circ p^*\bar{f}=\pi^*f\circ p$. We have, following our discussion in the paragraph above, that 
\[
(p^*\bar{f})^*\alpha:0\to p^*\pi^*\diffEl{P}.
\]
However, the pullback bundle has a natural trivialisation -- the diagonal section $\Delta:P\to\pi^*P$. Using this, and that $\Delta^*p^*=\mathrm{id}$, we see that
\[
\Delta^*(p^*\bar{f})^*\alpha:0\to \pi^*P.
\]
This is precisely the trivialisation $\delta_{\diffEl{P}}$ we encountered in Sec.~\ref{sec:difftwist}. In other words
\[
\delta_{\diffEl{P}}=\Delta^*(p^*\bar{f})^*\alpha.
\]
The following diagram should hopefully make everything clear:
\[
\xymatrix{
\pi^*P\ar[r]^{p^*\bar{f}}\ar[d]^p&ET\ar[d]^{\pi^u}\\
P\ar[r]^{\pi^*f}\ar[ur]^{\bar{f}}\ar@/^/@{-->}[u]^\Delta\ar[d]_\pi&BT\\
X\ar[ur]_f
}
\]

There is a second piece of natural structure on $\pi^*P$ -- it is a \emph{$T$-equivariant} bundle with equivariant connection with respect to the action of $T$ on $P$. There are \emph{two} commuting actions  of $T$ on $\pi^*P$, coming from $T$ acting on each factor of $\pi^*P=P\times_XP$. The action that makes $\pi^*P$ equivariant with respect to the $T$-action on $P$ is the action on the first factor; the action on the second factor is the action of $T$ as the structure group of the bundle. We will denote the equivariant action by putting bars on group elements, and the structure group action by omitting the bars. The statement, thus, that $\pi^*P\to P$ is an equivariant bundle reads
\[
p\circ \bar{t}=t\circ p.
\]
It is also easy to see that
\[
\Delta\circ t=t\circ\bar{t}\circ\Delta,\quad t^*(p^*\bar{f})^*=(p^*\bar{f})^*t^*,\quad \bar{t}^*(p^*\bar{f})^*=p^*\bar{f}.
\]
We now wish to compare $t^*\Delta_P$ and $\Delta_P$. We calculate from the above
\[
t^*\delta_{\diffEl{P}}-\delta_{\diffEl{P}}=\Delta^*(p^*\bar{f})^*(t^*\alpha-\alpha).
\]
But, 
\[
t^*\alpha-\alpha:0\to (t^*c-c)=0\to 0
\]
and thus $t^*\alpha-\alpha\in \diffEl{H}^1(ET;\Lattice)$. In fact, $ET$ is contractible, so $[t^*\alpha-\alpha]\in H^1(ET)$ is zero. This implies that $t^*\alpha-\alpha$ is geometrically trivialisable, and, once and for all, we fix a family of trivialisations 
\[
\eta_t:0\to t^*\alpha-\alpha.
\]
We note that both $\pi^*\sigma$ and $\pi^*\dual{P}$ are invariant of the action of $T$ on $P$ (where we recall $\sigma\in\Mor{H^4(X;\Z)}$ was used to build the twist), and thus, recalling
\[
\twist=\pi^*\sigma-\delta_{\diffEl{P}}\cdot\pi^*\dual{\diffEl{P}},
\]
we see that
\begin{align*}
t^*\twist-\twist&=t^*(\pi^*\sigma-\delta_{\diffEl{P}}\cdot\pi^*\diffEl{\dual{P}})-(\pi^*\sigma-\delta_{\diffEl{P}}\cdot\pi^*\diffEl{\dual{P}})\\
&=-(t^*\delta_{\diffEl{P}}-\delta_{\diffEl{P}})\cdot\pi^*\dual{\diffEl{P}}\\
&=-\Delta^*(p^*\bar{f})^*(t^*\alpha-\alpha)\cdot\pi^*\dual{\diffEl{P}}.
\end{align*}
Therefore we obtain morphisms
\[
-\Delta^*(p^*\bar{f})^*\eta_t\cdot\pi^*\dual{\diffEl{P}}:0\to(t^*\twist-\twist)
\]
and which, in turn, allow us to define morphisms $\theta_t:t^*\twist\to \twist$.

Using the morphisms $\theta_t$, one may define the \emph{fixed} subgroup of $\diffGroup{K}^\tau(X)$ as $x\in\diffGroup{K}^\tau(X)$ such that
\[
x=\theta_t(t^*x)
\]
for all $t\in T$.
\subsection{$T$-duality and invariant elements}\label{appendix:sec:tinv}
In this section we argue that the element $x_\poincare\in K^0(T^2)$, where $T^2=\R^2/\Z^2$, with 
\[
\ch x_\poincare=\left[1-\frac{i}{2\pi}\vol{T^2}\right]
\]
has no $T^2$-fixed pre-image in $\diffGroup{K}^0(T^2)$. However, we note that it does have a representative whose curvature character form \emph{is} invariant. This shows that the geometrically invariant subgroup of $\diffGroup{K}^0(T^2)$ is larger than the fixed subgroup, and in particular, that any statement of $T$-duality restricting attention only to the fixed subgroup of $\diffGroup{K}^0(T^2)$ would be too weak.

The class $x_\poincare$ is explicitly represented by a line bundle $L_\poincare\to T^2$ which we now construct. Let $\tilde{L}_\poincare\to\R\times \R$ be the trivial line bundle. Let $\Z\times\Z$ act on $\tilde{L}$ as follows:
\[
(n,m):(\theta_1,\theta_2,\xi)\mapsto(\theta_1+n,\theta_2+m,e^{2\pi i(n\theta_2+m\theta_1)}\xi),
\]
where $(n,m)\in\Z\times\Z$, $\theta_1,\theta_2\in\R$, and $\xi\in\C$. 
This makes $\tilde{L}_\poincare$ into a $\Z\times\Z$-equivariant line bundle. We may endow $\tilde{L}_\poincare$ with an equivariant connection given by
\[
\nabla_{\tilde{L}_\poincare}=\ud+\pi i(\theta_1\ud\theta_2-\theta_2\ud\theta_1).
\]
 We check that 
\[
\mathrm{Curv}\;\nabla_{\tilde{L}_\poincare}=2\pi i\,\ud\theta_1\ud\theta_2.
\]
We now define $L_\poincare$ as $\tilde{L}_\poincare/\Z\times\Z$. It inherits a connection $\nabla_{L_\poincare}$ from $\nabla_{\tilde{L}_\poincare}$, and, explicitly calculating the Chern character form shows that indeed
$[L_\poincare]=x_\poincare$.

We now consider the class\footnote{For the purposes of this appendix we will use the model of $\diffGroup{K}^0$ to be found in Klonoff's thesis \cite{Klonoff}. In his model classes in differential $K$-theory are formal differences of triples $[V,\nabla,\omega]$, where $V$ is a vector bundle, $\nabla$ a connection on it, and $\omega$ an odd degree differential form.} $\check{x}_\poincare=[L_\poincare,\nabla_{L_\poincare},0]\in\diffGroup{K}^0(T^2)$. A small computation shows that
\[
\theta_1^*\check{x}_\poincare-\check{x}_\poincare=i(\CS{\theta_1^*\nabla_{\tilde{L}_\poincare}}{\nabla_{\tilde{L}_\poincare}})=i(\pi i\,\theta_1\ud\theta_2).
\]
For generic $\theta_1$, $\CS{\theta_1^*\nabla_{\tilde{L}_\poincare}}{\nabla_{\tilde{L}_\poincare}}$ is not exact, so its image by $i$ is certainly non-zero in $\diffGroup{K}^0(T^2)$.

Any element of $\diffGroup{K}^0(T^2)$ that maps to $x_\poincare\in K^0(T^2)$ must be of the form $\check{x}_\poincare+[\alpha]$, where $[\alpha]\in\ker\kmap$ (and $\kmap$ is as in Eq.~\ref{eq:diffkexact}). We thus see a generic element in the pre-image of $x_\poincare$ is given by $\check{x}_\poincare+\formsmap(\alpha)$, for $\alpha\in\Forms^{1}(T^2)$. We compute 
\[
\theta_1^*(\check{x}_\poincare+\formsmap(\alpha))-(\check{x}_\poincare+i(\alpha))=i(\pi i(\theta_1\ud\theta_2)+\theta_1^*\alpha-\alpha).
\]
But there is no differential form $\alpha\in\Forms^1(T^2)$ such that
\[
\pi i(\theta_1\ud\theta_2)=\theta_1^*\alpha-\alpha+\im\ch K^{-1}(T^2)
\]
for all $\theta_1$.  Thus there is no fixed element of $\diffGroup{K}^0(T^2)$ in the pre-image of $x_\poincare\in K^0(T^2)$.

In fact, by noticing that a principal $U(1)$ bundle with connection is determined up to equivalence by its holonomy around every loop, we see that for any principal $T$-bundle $P\to X$, the only line bundles that give rise to fixed classes in $\diffGroup{K}^0(P)$ are those pulled back from $X$. Similar reasoning shows that the fixed classes in $\diffGroup{K}^0(P)$ are those with basic curvature.
\bibliographystyle{amsplain}
\bibliography{tdualpaper}
\end{document}